\documentclass[12pt]{amsart}

\usepackage{comment}
\usepackage[normalem]{ulem}
\usepackage{hyperref}
\hypersetup{
    colorlinks=true,
    linkcolor=Mahogany,      
    urlcolor=cyan,
    citecolor=blue,
    }
\usepackage[abbrev,alphabetic]{amsrefs}

\usepackage{amscd} 
\usepackage{amsmath}
\usepackage{mathtools}
\usepackage{amssymb} 
\usepackage{amsthm}
\usepackage{bigdelim}
\usepackage{color} 
\usepackage{enumerate}
\usepackage{graphicx}
\usepackage{mathrsfs}
\usepackage{multirow}
\usepackage{enumitem}
\usepackage[dvipsnames]{xcolor}
\usepackage{tikz-cd}
\usepackage{color,soul}

\newtheorem{thm}{Theorem}[section] 

\newtheorem{cor}[thm]{Corollary}
\newtheorem{prop}[thm]{Proposition}
\newtheorem{conj}[thm]{Conjecture}

\theoremstyle{definition} 

\newtheorem{eg}[thm]{Example}

\theoremstyle{remark}
\newtheorem{rem}[thm]{Remark}

\numberwithin{equation}{section}

\newcommand{\Z}{\mathbb{Z}}
\newcommand{\C}{\mathbb{C}}
\newcommand{\Q}{\mathbb{Q}}
\newcommand{\mP}{\mathbb{P}}

\title[]
{Lagrangian Fibrations onto Varieties with Isolated Quotient Singularities}

\author{Niklas M\"uller, Zheng Xu}
\address{Department of Mathematics, Universit\"at Duisburg-Essen,
Thea-Leymann-Str. 9, 45127 Essen, Germany.}
\email{niklas.mueller@uni-duisburg-essen.de}
\address{Beijing International Center for Mathematical Research,
Peking University, No. 5 Yiheyuan Road, Haidian District, Beijing 100871, China}
\email{zhengxu@pku.edu.cn}

\date{\today}



\begin{document}

\begin{abstract}

In this note, we show that if $f\colon M\rightarrow X$ is a germ of a projective Lagrangian fibration from a holomorphic symplectic manifold $M$ onto a normal analytic variety $X$ with isolated quotient singularities, then $X$ is smooth. In particular, if $f\colon M\rightarrow X$ is a Lagrangian fibration from a hyper-K\"ahler fourfold $M$ onto a normal surface $X$, then $X\cong \mathbb{P}^2$, which recovers a recent result of Huybrechts--Xu and Ou.
\end{abstract}

\maketitle

\tableofcontents

\section{Introduction}

Let $M$ be a complex manifold. A \emph{holomorphic symplectic form} on $M$ is an element $\sigma \in H^0(M, \Omega^2_M)$ such that $\sigma$ is $d$-closed and the associated bilinear pairing $\sigma\colon \mathcal{T}_M \times \mathcal{T}_M \rightarrow \mathcal{O}_M$ is everywhere non-degenerate.
We refer to the pair $(M, \sigma)$ as a \emph{holomorphic symplectic manifold}. 
It is well-known that in this case $\dim M =2n$ is an even integer and that $\omega_M\cong \mathcal{O}_M$ is trivial. 
A holomorphic symplectic manifold $M$ is called \emph{hyper-K\"ahler} if it is a simply connected compact K\"ahler manifold.
Hyper-K\"ahler manifolds form one of the fundamental building blocks of varieties with trivial canonical bundle, as established by the Beauville–Bogomolov decomposition \cites{Bog74,Bea83}. 

A complex subvariety $Z$ of a holomorphic symplectic manifold $(M, \sigma)$ is called \emph{Lagrangian} if $\dim Z = n$ and the restriction of $\sigma$ to the smooth locus $Z^{\mathrm{reg}}$ vanishes identically, i.e., $\sigma|_{Z^{\mathrm{reg}}}=0 \in H^0(Z^{\mathrm{reg}}, \Omega^2_{Z^{\mathrm{reg}}})$. A proper fibration $f\colon M \rightarrow X$ onto a normal analytic variety $X$ is called \emph{Lagrangian} if any irreducible component of any fibre of $f$ is Lagrangian. In particular, Lagrangian fibrations are equidimensional. According to a result of Matsushita \cites{Mat99,Mat99b}, any fibration $f\colon M \rightarrow X$ from a hyper-K\"ahler manifold onto a normal analytic variety $X$ with $0<\mathrm{dim}\ X< \mathrm{dim}\ M$ is Lagrangian.

Hyper-K\"ahler manifolds and their Lagrangian fibrations have been studied from various angles in the past years, see \cite{HM22} for a comprehensive survey. In this note, we focus on properties of the base space $X$ of a projective Lagrangian fibration $f\colon M \rightarrow X$ from a holomorphic symplectic manifold of dimension $2n$. Matsushita \cite{Mat99} proved that if $M$ is a projective hyper-K\"ahler manifold and $X$ is projective, then $X$ is a klt Fano variety of Picard number one. Furthermore, if $X$ is smooth, he showed that the Hodge structure of $X$ matches that of $\mathbb{P}^n$ \cite{Mat05}. This result was generalized by Shen and Yin \cite{SY22} to the case where $M$ is hyper-K\"ahler and $X$ is not necessarily smooth. These results lead to the following conjecture:

\begin{conj}
\label{conj-base-is-Pn}
    Let $M$ be a hyper-K\"ahler manifold of dimension $2n$, and let $f\colon M \rightarrow X$ be a projective Lagrangian fibration. Then $X \cong \mathbb{P}^n$.
\end{conj}

Conjecture \ref{conj-base-is-Pn} was proved by Hwang \cite{Hwa08} under the assumption that $X$ is smooth and $M$ is projective. The projectivity assumption was removed in \cite{GL14}, see also \cites{BS23,LT24} for recent alternative proofs. 
Moreover, Conjecture \ref{conj-base-is-Pn} is verified for all currently known examples of hyper-K\"ahler manifolds; namely when $M$ is of $K3^{[n]}$ or of generalized Kummer type \cite{BM14, Mar14, Yos16}, when $M$ is of OG6 type \cite{MR21} and when $M$ is of OG10 type \cite{MO22}; c.f.\ also \cite{DHMV24} for an intelligible summary. Finally, the conjecture is known to hold in the case where $M$ is projective of dimension four by Huybrechts--Xu \cite{HX22}, building on earlier work by Ou \cite{Ou19}:

\begin{thm}\emph{(Ou \cite{Ou19}, Huybrechts--Xu \cite{HX22})}
\label{thm-HX}

    \noindent
    Let $f\colon M \rightarrow X$ be a Lagrangian fibration 
    from a projective hyper-K\"ahler manifold $M$ of dimension four onto a normal projective surface $X$. Then $X\cong \mathbb{P}^2$.
\end{thm}
More precisely, based on the classification of canonical Fano surfaces by Keel--McKernan \cite{KMK99}, Ou \cite{Ou19} showed that $X$ is either smooth, and thus isomorphic to $\mathbb{P}^2$, or has a unique quotient singularity of type $\mathrm{E}_8$. Huybrechts--Xu \cite{HX22} subsequently proved that if $f\colon M \rightarrow X$ is a germ of a projective Lagrangian fibration from a holomorphic symplectic fourfold $M$ onto a normal surface $X$, then $X$ can not have $\mathrm{E}_8$ singularities, thereby completing the proof of Theorem \ref{thm-HX}; see also \cite{BK18} for an independent proof in case the central fibre of $f$ is not multiple. The methods in \cite{HX22} are completely local in nature and rely on deep results of Halle--Nicaise \cite{HN18}, which provide a classification of the essential skeletons of semiabelian degenerations of abelian surfaces.

The main result of this work, presented below, extends the main theorem of \cite{HX22} by proving that the base of a projective Lagrangian fibration can \emph{never} admit an isolated quotient singularity:

\begin{thm}
\label{thm-main}
    Let $f\colon M \rightarrow X$ be a germ of a projective Lagrangian fibration from a holomorphic symplectic manifold $M$ onto a normal analytic variety $X$ with only isolated quotient singularities. Then $X$ is smooth.
\end{thm}
Note that the assumption that $X$ has quotient singularities is very natural in light of a conjecture of Gurjar \cite{Gur03}.
\begin{cor}
\label{cor-open-LF-n=2}
    Let $f\colon M \rightarrow X$ be a germ of a projective Lagrangian fibration from a holomorphic symplectic manifold $M$ of dimension four onto a normal analytic surface $X$. Then $X$ is smooth.
\end{cor}

As an immediate corollary, we obtain an alternative proof of Theorem \ref{thm-HX}, which avoids the use of \cites{KMK99, HN18}. The conclusion of Theorem \ref{thm-main} is somewhat surprising to the authors, as
there are well-known examples which show that, in the non-compact setting, $X$ may indeed have (non-isolated) quotient singularities, see Section \ref{section_sharpness}.

The basic strategy in the proof of Theorem \ref{thm-main} is similar to that in \cite{BK18, HX22}: locally, one writes $X=X'/G$, where $X'$ is a smooth variety and $G$ is a finite group. Let $M'$ denote the normalisation of the fibre product $M\times_X X'$. It follows from purity of the branch locus, that $M'$ is smooth and that the induced morphism $\pi\colon M'\rightarrow M$ is \'etale. Then the idea in \cite{HX22} is to find a fixed point for the action of $G$ on $M'$, which would contradict the \'etaleness of $\pi$, thereby completing the proof. Indeed, in the setting of Theorem \ref{thm-main}, this contradiction is established in Section \ref{sec-proofs} by applying the singular Lefschetz-Riemann–-Roch formula of Baum--Fulton--Quart \cite{BFQ_SingularLefschetzRiemannRoch}, together with a short computation that uses Matsushita's formula for the higher direct images of the structure sheaf under Lagrangian fibrations \cite{Mat05}.

With slightly more effort, we can also treat the case that $\dim M = 6$, assuming $X$ has only terminal singularities:
\begin{thm}
\label{thm-open-LF-n=3}
     Let $f\colon M \rightarrow X$ be a germ of a projective Lagrangian fibration from a holomorphic symplectic manifold $M$ of dimension six onto a normal analytic threefold $X$. Assume that the singularities of $X$ are terminal. Then $X$ is smooth.
\end{thm}

In fact, our argument for Theorem \ref{thm-open-LF-n=3} shows slightly more, see Remark \ref{rem_klt isolated case}.

\section*{Acknowledgments}
\noindent

The authors would like to thank Jihao Liu for inviting N.M. to visit him at Peking University, where the authors met and this project was initiated. The authors are grateful to Franco Giovenzana, Christian Lehn, Stefan Kebekus, Wenhao Ou, Chenyang Xu, Qizheng Yin and Vanja Zuliani for fruitful discussions on Lagrangian fibrations, to Hyunsuk Kim and Thomas Peternell for helpful discussions on singularities and to Marc Levine for explanations concerning the paper \cite{BFQ_SingularLefschetzRiemannRoch}. Additionally, N.M. wishes to thank Andreas H\"oring for organising a seminar on the paper \cite{Ou19}. Finally, the authors are grateful to the anonymous referees for their valuable suggestions.

N.M. gratefully acknowledges financial support by the DFG Research Training Group 2553, ``Symmetries and classifying spaces: analytic, arithmetic and derived" and is grateful to Daniel Greb for his valuable advise. Z.X. would like to thank Zhiyu Tian for his comments and constant support.

\section{Preliminary results}\label{sec-pre}

\subsection{Notation and Conventions}\label{subsec-notation}

Throughout this paper, we work over the field $\C$ of complex numbers.
We employ the standard notation and conventions in \cites{Har77, KM98}. If $X$ is a normal complex analytic variety, we denote by $\Omega_X^1$ its sheaf of Kähler differentials and by $\Omega_X^{[1]} := (\Omega_X^1)^{\vee\vee}$ the reflexive hull. Throughout this paper, by an action of a group we will always mean an action from the left.

\subsection{Preliminaries on Lagrangian Fibrations}
We start this section by collecting some properties of Lagrangian fibrations. Although the precise statements that we will require later might not necessarily be available in the current literature, the results presented below certainly stick very closely to what is known.
\begin{prop}
\label{prop-base-space-equidim-fibration}

    Let $f\colon Y \rightarrow X$ be a germ of a proper equidimensional fibration from a complex manifold $Y$ onto a normal analytic variety $X$. Then 
    \begin{itemize}
        \item[\emph{(1)}] $X$ is $\Q$-factorial and log terminal.
        \item[\emph{(2)}] The sheaves $\Omega_X^{[i]}$ are Cohen-Macaulay for any $i\geq 0$.
    \end{itemize}
\end{prop}
\begin{proof}
    As in the proof of \cite[Prop. 1.10]{HM22}, locally, there exists a complex submanifold $Z\subseteq Y$ such that the restriction $f|_Z\colon Z \rightarrow X$ is finite. Replacing $Y$ by $Z$, we may assume that $f$ is finite. This immediately implies that $X$ is $\Q$-factorial \cite[Lem.\ 5.16]{KM98}. Let $D_f$ be the ramification divisor. Then $K_Y - D_f \sim f^*K_X$. Since the pair $(Y, -D_f)$ is klt \cite[Cor.\ 2.35]{KM98}, we deduce that $X$ is log terminal \cite[Cor.\ 2.43]{Kol13}. This shows (1). In particular, $X$ is Cohen-Macaulay.

    Concerning (2), by \cite[Prop.\ 4.2]{SVV23}, $X$ has pre-$k$-rational and pre-$k$-Du Bois singularities for any $k\geq 0$. By definition \cite[Def.\ 1.1]{SVV23}, this means that the $i$-th Du Bois complex $\underline{\Omega}^i_X$ is concentrated in degree zero for all $i\geq 0$ and that
    \begin{equation}
         \mathcal{E}xt^{j}_{\mathcal{O}_X}\big(\underline{\Omega}^i_X, \omega_X\big) = 0, \qquad \forall j>0, \ \forall i\geq 0.
         \label{eq-CM}
    \end{equation}
    Since $X$ is log terminal, $\underline{\Omega}^i_X = \mathcal{H}^0(\underline{\Omega}^i_X) \cong \Omega_X^{[i]}$, see \cite[Rmk.\ 2.5]{SVV23}. Consequently, it follows from \eqref{eq-CM}, that $\Omega_X^{[i]}$ is Cohen-Macaulay for any $i\geq 0$, see \cite[Cor.\ 5.3]{BS76}.
\end{proof}

\begin{rem}\label{rem_quotient}
    More generally, it is conjectured \cite{Gur03}, see also \cite[2.24]{Kol07}, that in the situation of Prop.\ \ref{prop-base-space-equidim-fibration}, $X$ should have finite quotient singularities. This is known in case $\dim X = 2$ \cite{Mum61, Brie68}. 
\end{rem}

\begin{prop}\emph{(Takegoshi \cite{Tak95})}
\label{prop-CM}

\noindent
    Let $f\colon M \rightarrow X$ be a germ of a projective Lagrangian fibration from a holomorphic symplectic manifold $M$ onto a normal analytic variety $X$. Then the sheaves $R^if_*\mathcal{O}_M$ are reflexive for all $i\geq 0$.
\end{prop}
\begin{proof}
    Using that $\omega_M \cong \mathcal{O}_M$, this is just a special case of \cite[Thm.\ 6.5.($\beta$)]{Tak95}.
\end{proof}

Without the equivariance assertion, the following result is well-known:
\begin{cor}\emph{(Matsushita \cite{Mat05}, Ou \cite[Prop.\ 3.16]{Ou19})}
    \label{cor-matsushita}

    \noindent
    Let $f\colon M \rightarrow X$ be a germ of a projective Lagrangian fibration from a holomorphic symplectic manifold $(M, \sigma)$ onto a normal analytic variety $X$. Let $\mathcal{L}$ be an $f$-ample line bundle. Then there exist natural isomorphisms
    \[
    \theta^i_{\mathcal{L}} \colon R^if_*\mathcal{O}_M \rightarrow \Omega_X^{[i]}, \qquad \forall i\geq 0.
    \]
    In particular, if $G$ is a group acting equivariantly on $f$ such that the action of $G$ on $M$ is symplectic and preserves $c_1(\mathcal{L})$, then the isomorphisms $\theta^i_{\mathcal{L}}$ are $G$-equivariant for the natural actions of $G$ on either side.
\end{cor}
\begin{proof}
    We follow the proof in \cite[Thm.\ 1.2]{Mat05}: Let $U\subseteq X^{\mathrm{reg}}$ be the maximal Zariski open subset such that $f\colon V := f^{-1}(U) \rightarrow U$ is smooth. Since $f$ is Lagrangian, we have a commutative diagram
    \[
    \begin{tikzcd}
            0 
            \arrow{r} 
            &
            \mathcal{T}_{V/U}
            \arrow{r}
            \arrow[dashrightarrow]{d}{\sigma}
            &
            \mathcal{T}_V
            \arrow{r}
            \arrow{d}{\sigma}
            & 
            f^*\mathcal{T}_U
            \arrow{r}
            \arrow[dashrightarrow]{d}{\sigma}
            &  
            0
            \\
            0 
            \arrow{r} 
            &
            f^*\Omega_U^1
            \arrow{r}
            &
            \Omega_V^1
            \arrow{r}
            & 
            \Omega_{V/U}^1
            \arrow{r}
            & 
            0 
        \end{tikzcd}
    \]
    in which all vertical arrows are isomorphisms. In particular,
    \begin{equation}
        \Omega^1_X\big|_U 
        \cong \mathcal{T}_U^{\vee} 
        \cong \Big(f_*f^*\mathcal{T}_U\Big)^{\vee} 
        \cong \Big(f_*\Omega_{V/U}\Big)^\vee.
        \label{eq-matsushita}
    \end{equation}
    Note that \eqref{eq-matsushita} is equivariant with respect to the action of any group acting equivariantly on $f$ and symplectically on $M$. Additionally, by relative Hard Lefschetz, $c_1(\mathcal{L})$ determines a natural isomorphism
    \begin{equation}
        \Big(f_*\Omega_{V/U}\Big)^\vee \cong R^1f_*\mathcal{O}_M\big|_U,
        \label{eq-matsushita-2}
    \end{equation}
    which is equivariant with respect to the action of any group acting equivariantly on $f$ and which preserves $c_1(\mathcal{L})$. In effect, we obtain a natural isomorphism
    \[
    \theta^1_{\mathcal{L}|_U}\colon R^1f_*\mathcal{O}_M\big|_U \rightarrow \Omega_X^1\big|_U.
    \]
    Since the fibres of $f$ are Abelian varieties \cite[Lem.\ 1.5]{HM22}, also
    \begin{equation}
        \theta^i_{\mathcal{L}|_U} := \wedge^i \theta^1_{\mathcal{L}|_U}\colon R^if_*\mathcal{O}_M\big|_U \cong \wedge^i R^1f_*\mathcal{O}_M\big|_U\rightarrow \Omega_X^i\big|_U
        \label{eq-matsushita-3}
    \end{equation}
    is an isomorphism for any $i\geq 0$ and is satisfies the required equivariance property since so do \eqref{eq-matsushita} and \eqref{eq-matsushita-2}. Matsushita has shown \cite[Lem.\ 3.1., Prop.\ 3.3]{Mat05} that the isomorphisms \eqref{eq-matsushita-3} extend to isomorphisms
    \begin{equation}
        \theta^i_{\mathcal{L}|_{X^{\mathrm{reg}}}}\colon R^if_*\mathcal{O}_M\big|_{X^{\mathrm{reg}}} \rightarrow \Omega_{X^{\mathrm{reg}}}^i.
        \label{eq-matsushita-4}
    \end{equation}
    Strictly speaking, \cite[Lem.\ 3.1., Prop.\ 3.3]{Mat05} is formulated only in case $X$ is projective but the proof is completely local in the analytic topology. In any case, \eqref{eq-matsushita-4} extends uniquely to a isomorphisms
    \begin{equation}
        \theta^i_{\mathcal{L}}\colon R^if_*\mathcal{O}_M \rightarrow \Omega_{X}^{[i]},
        \label{eq-matsushita-5}
    \end{equation}
    since the sheaves in \eqref{eq-matsushita-5} are both reflexive, see Prop.\ \ref{prop-CM}. Then $\theta^i_{\mathcal{L}}$ is automatically equivariant with respect to any group acting equivariantly on $f$ and preserving $\sigma$ and $c_1(\mathcal{L})$, since both sheaves are torsion-free and since $\theta^i_{\mathcal{L}}\big|_U$ is equivariant.
\end{proof}

\subsection{Baum--Fulton--Quart Lefschetz-Riemann--Roch formula}\label{subsection_singular LRR}

The following result is the main ingredient in our proof of Theorem \ref{thm-main}:
\begin{thm}\emph{(Singular Lefschetz-Riemann--Roch, Baum--Fulton--Quart \cite{BFQ_SingularLefschetzRiemannRoch})}
    \label{thm-SingularLefschetzRiemannRoch}
    Let $Y$ be a projective scheme and let $\phi\colon Y \rightarrow Y$ be an automorphism of finite order. If the fixed point scheme $Y^\phi = \emptyset$ is empty, then
    \[
    \sum_{i} (-1)^i\ \mathrm{Tr}\big(\phi^*|_{H^i(Y, \mathcal{O}_Y)}\big) = 0.
    \]
\end{thm}
Note that Baum--Fulton--Quart call a pair $(Y, \phi)$, where $Y$ is a projective scheme $Y$ and $\phi\colon Y \rightarrow Y$ is an automorphism of finite order an \emph{equivariant variety}, even thought $Y$ need not be reduced or irreducible, cf.\ \cite[Introduction]{BFQ_SingularLefschetzRiemannRoch}. The fixed point scheme $Y^\phi$ is defined to be the fibre product
\[
        \begin{tikzcd}
            Y^\phi
            \arrow{r} 
            \arrow{d}
            & 
            Y
            \arrow{d}{\Delta} 
            \\
            Y
            \arrow{r}{(\mathrm{id}, \phi)} 
            &
            Y\times Y.
        \end{tikzcd}
\]
In this situation, \cite{BFQ_SingularLefschetzRiemannRoch} construct a natural group homomorphism
\begin{equation}
    L_Y\colon K_0^{\mathrm{eq}}(Y) \rightarrow K_0^{\mathrm{abs}}(Y^\phi)\otimes_\Z \C.
    \label{eq-L}
\end{equation}
Here, $K_0^{\mathrm{eq}}(Y)$ is the Grothendieck group of equivariant sheaves on $Y$, i.e.\ pairs $(\mathcal{F}, \alpha)$, where $\mathcal{F}$ is a coherent sheaf on $Y$ and $\alpha\colon \mathcal{F} \rightarrow \phi^*\mathcal{F}$ is an $\mathcal{O}_{Y}$-linear map \cite[0.2]{BFQ_SingularLefschetzRiemannRoch} and $K_0^{\mathrm{abs}}(Y^\phi)$ denotes the ordinary $K$-groups of $Y^\phi$. Concretely, in case $Y = \{\ast\}$ is a point, $K_0^{\mathrm{eq}}(\{\ast\})$ is the Grothendieck group of finite dimensional $\C$-vector spaces with an endomorphism, $K_0^{\mathrm{abs}}( \{ \ast \})\otimes_\Z \C \cong \C$ and
\begin{equation}
    L_{\{ \ast \}}\colon K_0^{\mathrm{eq}}(\{\ast\}) \rightarrow \C, \qquad (V, \alpha) \mapsto \mathrm{Tr}\big(\alpha|_V\big),
    \label{eq-L-star}
\end{equation}
see \cite[0.4]{BFQ_SingularLefschetzRiemannRoch}. Naturality of \eqref{eq-L} means, that $L_\bullet$ is covariant with respect to proper morphisms.

\begin{proof}[Proof of Thm.\ \ref{thm-SingularLefschetzRiemannRoch}]
    Let $p\colon Y \rightarrow \{\ast\}$ be the structure morphism. By \cite[Main Result]{BFQ_SingularLefschetzRiemannRoch}, the diagram
    \begin{equation}
        \begin{tikzcd}
            K_0^{\mathrm{eq}}(Y)
            \arrow{r}{L_Y} 
            \arrow{d}{p_*}
            & 
            K_0^{\mathrm{abs}}(Y^\phi)\otimes_\Z \C
            \arrow{d}{p_*} 
            \\
            K_0^{\mathrm{eq}}(\{\ast\})
            \arrow{r}{L_{\{ \ast \}}}
            &
            K_0^{\mathrm{abs}}( \{ \ast \})\otimes_\Z \C.
        \end{tikzcd}
        \label{eq-BFQ}
    \end{equation}
    commutes. We consider the element $(\mathcal{O}_Y, \phi^*) \in K_0^{\mathrm{eq}}(Y)$. As $Y^\phi = \emptyset$, we infer from the commutativity of \eqref{eq-BFQ} that
    \[
    0 = L_{\{ \ast \}}\Big( p_*\big(\mathcal{O}_Y, \phi^* \big) \Big) = \sum_i (-1)^i \ \mathrm{Tr}\big(\phi^*|_{H^i(Y, \mathcal{O}_Y)}\big).
    \]
    Here, the second equality follows from \eqref{eq-L-star} and \cite[0.2]{BFQ_SingularLefschetzRiemannRoch}
\end{proof}

\section{Proof of the main results}
\label{sec-proofs}

In this section we prove Thm. \ref{thm-main}, Cor.\ \ref{cor-open-LF-n=2} and Thm.\ \ref{thm-open-LF-n=3}. In fact, these results are all immediate consequences of the following formula:
\begin{thm}
\label{thm-main-formula}
    Let $f\colon M \rightarrow X$ be a germ of a projective Lagrangian fibration from a holomorphic symplectic manifold $(M, \sigma)$ onto a complex manifold $X$. Let $G$ be a finite group acting equivariantly on $f$. Assume that the action of $G$ on $M$ is symplectic and fixed point free. If $x\in X$ is a fixed point for the action of $G$ on $X$, then
    \begin{equation}
         \det\big(\mathrm{id}_{T_xX} - dg|_x\big) = 0, \qquad \forall g\in G.
         \label{eq-main-formula}
    \end{equation}
\end{thm}
\begin{proof}
    Denote $n := \dim X$ and let $\mathcal{L}$ be an $f$-ample line bundle on $M$. Replacing $\mathcal{L}$ by 
    \[
    \bigotimes_{g\in G} g^*\mathcal{L},
    \]
    we may assume that $c_1(\mathcal{L})$ is $G$-invariant. Let $F:=f^{-1}(x)$ be the scheme-theoretic fibre, which is a projective scheme. We claim that the natural morphisms
    \begin{equation}
        R^if_*\mathcal{O}_M\big|_x \rightarrow H^i(F, \mathcal{O}_F)
        \label{eq-basechange}
    \end{equation}
    are $G$-equivariant isomorphisms for any $i\geq 0$. Indeed, the sheaves $R^if_*\mathcal{O}_M \cong \Omega_X^i$ are locally free by Cor.\ \ref{cor-matsushita}, and it follows from Cohomology and Base Change \cite[Thm.\ III.3.4, Cor.\ III.3.10]{BS76} that \eqref{eq-basechange} is an isomorphism for any $i$. Here, we used that $f$ is flat since it is equidimensional and since $M$ and $X$ are smooth, see \cite[p.\ 154]{Fis76}. The $G$-equivariance follows from Cor.\ \ref{cor-matsushita}. Note that the natural (left) action of $G$ on either side is given by
    \[
    \big(g^{-1}\big)^*\colon H^i(F, \mathcal{O}_F) \rightarrow H^i(F, \mathcal{O}_F), 
    \quad \mathrm{resp.\ } \quad
    \big(g^{-1}\big)^*\colon \Omega_X^i\big|_x \rightarrow \Omega_X^i\big|_x,
    \quad \forall g\in G.
    \]
    It follows that
    \begin{align}
        \begin{split}
           \sum_{i=0}^n (-1)^i\ \mathrm{Tr}\Big( \big(g^{-1}\big)^*|_{H^i(F, \mathcal{O}_F)} \Big)
            & = \sum_{i=0}^n (-1)^i\ \mathrm{Tr}\Big( \big(g^{-1}\big)^*|_{\Omega_X^i\big|_x}\Big) \\
            & = \sum_{i=0}^n (-1)^i\ \mathrm{Tr}\Big( \wedge^i dg|_x\Big)\\
            & = \det\big(\mathrm{id}_{T_xX} - dg|_x\big),
        \end{split}
        \label{eq-formula-1}
    \end{align}
    where we used the well-known formula
    \[
    \sum_{i=0}^n (-1)^i\ \mathrm{Tr}\Big( \wedge^i\phi\Big)
    = \det\big(\mathrm{id}_{V} - \phi\big),
    \]
    which is valid for any endomorphism $\phi \colon V \rightarrow V$ of an $n$-dimensional vector space $V$. Finally,
    \begin{equation}
        \sum_{i=0}^n (-1)^i\ \mathrm{Tr}\Big( \big(g^{-1}\big)^*|_{H^i(F, \mathcal{O}_F)} \Big) = 0
        \label{eq-formula-2}
    \end{equation}
    by Thm.\ \ref{thm-SingularLefschetzRiemannRoch}. Indeed, by the universal property of fibre products, for any $g\in G$ we have a natural closed immersion $F^g \hookrightarrow M^g$. As $M^g=\emptyset$, we deduce that $F^g = \emptyset$ for any $g\neq 1$, as required. Combining \eqref{eq-formula-1} and \eqref{eq-formula-2} yields the sought-after formula \eqref{eq-main-formula}.
\end{proof}

\subsection{Proof of Theorem \ref{thm-HX}, Theorem \ref{thm-main} and Corollary \ref{cor-open-LF-n=2}}

\begin{proof}[Proof of Thm.\ \ref{thm-main}]
    Let $x\in X$ be an isolated quotient singularity. By \cite{Prill67}, up to shrinking $X$, there exists a finite subgroup $G\subseteq \mathrm{GL}_n(\C^n)$ without reflections and a $G$-invariant analytic open neighborhood $X'\subseteq \C^n$ such that $X\cong X'/G$. In fact, since $x$ is an isolated singularity, $(\C^n)^g = \{0\}$ for any element $g\in G\setminus \{1\}$. Equivalently,
    \begin{equation}
        \det\big(\mathrm{id}_{\C^n} - g\big) \neq 0, \qquad \forall g\in G\setminus \{1\}.
        \label{eq-det-1}
    \end{equation}
    Let $M'$ denote the normalisation of the fibre product $M\times_X X'$. The situation is summarised in the following diagram:
    \[
        \begin{tikzcd}
            M'
            \arrow{r}{\psi} 
            \arrow{d}{f'}
            & 
            M
            \arrow{d}{f} 
            \\
            X'
            \arrow{r}{\pi} 
            &
            X.
        \end{tikzcd}
    \]
    Since $G$ does not contain a reflection, $\pi$ is \'etale in codimension one. Since $f$ is equidimensional, also $\psi$ is \'etale in codimension one. By purity of the branch locus \cite[Thm.\ 12.3.3]{Nem22}, $\psi$ is \'etale and $M'$ is smooth. In particular, $M'$ is symplectic and $f'\colon M' \rightarrow X'$ is a Lagrangian fibration. Let $x' = 0\in X'$ be the unique point with $\pi(x') = x$. Then $x'$ is a fixed point for the action of $G$ on $X'$. By Theorem \ref{thm-main-formula},
    \begin{equation}
    \det\big(\mathrm{id}_{T_{x'}X'} - dg|_{x'} \big) = \det\big(\mathrm{id}_{\C^n} - g\big) = 0, \qquad \forall g\in G.
        \label{eq-det-2}
    \end{equation}
    From \eqref{eq-det-1} and \eqref{eq-det-2} together we infer that $G = \{1\}$, indicating that $X\cong X'$ is smooth.
\end{proof}

\begin{proof}[Proof of Cor.\ \ref{cor-open-LF-n=2}]
    As in the proof of \cite[Prop. 1.10]{HM22}, there exists a closed submanifold $Z\subseteq M$ such that restriction $f|_Z\colon Z \rightarrow X$ is finite. It follows that $X$ has (isolated) quotient singularities \cite[Satz 2.8]{Brie68}. Thus, the result follows from Thm.\ \ref{thm-main}.
\end{proof}

\begin{proof}[Proof of Thm.\ \ref{thm-HX}]
    By Cor.\ \ref{cor-open-LF-n=2}, $X$ is a smooth projective surface. Moreover, by \cite[Thm.\ 2]{Mat99}, $X$ is a Fano variety of Picard number one. It follows that $X\cong \mP^2$.
\end{proof}

\subsection{Proof of Theorem \ref{thm-open-LF-n=3}}

In this section we prove Thm.\ \ref{thm-open-LF-n=3}. We will require the following Proposition:

\begin{prop}\emph{(cf.\ \cite[Prop. 9.7]{Kunz86})}
\label{prop-Kunz}

\noindent
    Let $X$ be a complex analytic variety which is locally a complete intersection. 
    \begin{itemize}
        \item[\emph{(1)}] If $X$ is smooth in codimension two, then $\Omega_X^1$ is reflexive.
        \item[\emph{(2)}] If $\Omega_X^1$ is Cohen-Macaulay, then $X$ is smooth.
    \end{itemize}
\end{prop}
\begin{proof}
    In case $X$ is algebraic, this is a special case of \cite[Prop. 9.7]{Kunz86}. In the general analytic case, $(1)$ follows from \cite{Vet70}. Regarding $(2)$, let $x\in X$ be a point. Denote $A := \mathcal{O}_{X, x}$, which is a Noetherian local ring. By \cite[Prop. 9.7]{Kunz86}, $A$ is regular. Consequently, $X$ is smooth at $x$ \cite[Prop.\ 5.1, 5.2]{DG67}.
\end{proof}

\begin{proof}[Proof of Thm.\ \ref{thm-open-LF-n=3}]
    By a result of Reid \cite[Cor.\ 5.39]{KM98}, there exists a terminal threefold $X'$ with an isolated cDV singularity $x'\in X'$ and a finite group $G$ acting on $X'$ without fixed points on $X'\setminus \{x'\}$ such that $X\cong X'/G$. As in the proof of Thm.\ \ref{thm-main}, we denote by $M'$ the normalisation of the fibre product $M\times_X X'$. The situation is summarised in the following diagram:
    \[
        \begin{tikzcd}
            M'
            \arrow{r}{\psi} 
            \arrow{d}{f'}
            & 
            M
            \arrow{d}{f} 
            \\
            X'
            \arrow{r}{\pi} 
            &
            X.
        \end{tikzcd}
    \]
    By assumption, $\pi$ is \'etale away from $x$. Since $f$ is equidimensional, $\psi$ is \'etale in codimension one. By purity of the branch locus, $\psi$ is \'etale and $M'$ is smooth. In particular, $M'$ is symplectic and $f'\colon M' \rightarrow X'$ is a Lagrangian fibration. Now, recall that cDV singularities are hypersurface singularities \cite[Def.\ 2.1]{Reid80}. Since $X'$ is smooth in codimension two, $\Omega^1_{X'}$ is reflexive, see Prop.\ \ref{prop-Kunz}. By Prop.\ \ref{prop-base-space-equidim-fibration},
    \[
    \Omega_{X'}^1 = \Omega_{X'}^{[1]}
    \]
    is in fact a Cohen-Macaulay sheaf. Thus, $X'$ is smooth by Prop.\ \ref{prop-Kunz}. We conclude that $x\in X = X'/G$ is an isolated quotient singularity. Thus, the result follows from Thm.\ \ref{thm-main}.
\end{proof}

\begin{rem}\label{rem_klt isolated case}
    Observe that our proof applies more generally whenever $x\in X$ is an isolated threefold singularity which is a quotient of an cDV singularity. By similar arguments, to extend Thm.\ \ref{thm-open-LF-n=3} to the case of arbitrary isolated (log terminal) singularities, it remains to treat the case of isolated canonical Gorenstein singularities $x\in X$ such that a general hypersurface $x\in H \subseteq X$ has a minimal elliptic singularity at $x$, cf.\ \cite[Lem.\ 5.30]{KM98}. 
\end{rem}

\section{Examples}\label{section_sharpness}

\noindent
In this section, we provide some examples to show that Theorem \ref{thm-main} is sharp. Concretely, we demonstrate that the conclusion in Theorem \ref{thm-main} does not hold in either one of the following more general situations:
\begin{itemize}
    \item[(1)] $X$ has (not necessarily isolated) quotient singularities, or
    \item[(2)] $M$ is a singular symplectic variety.
\end{itemize}
\begin{eg}
    The following example shows that the conclusion in Theorem \ref{thm-main} is wrong if we only assume that $X$ has (not necessarily isolated) quotient singularities. The authors learnt of this example from an unpublished lecture note of Kurnosov \cite[p.\ 57]{Kur22}; they are grateful to the referee for pointing out that it goes back to a construction of Hwang and Oguiso \cite{HO11}.
    Indeed, let $E$ be an elliptic curve, denote $M':=E\times E\times E\times \C^3$, and let $f'\colon M'\to \C^3$ be the projection. Note that $M'$ admits an action by $G := \Z/2\Z$ as follows:
    $$[1]:(u,v,w,x,y,z)\mapsto (-u,-v,w+\tau,-x,-y,z),$$
    where $\tau$ is a $2$-torsion point of $E$. Let $M:=M'/G$ and $X:=\C^3/G$ be the respective quotients, and let $f\colon M\to X$ be the natural projection. Note that the action of $G$ on $M'$ has no fixed points, so that $M$ is smooth. Note, moreover, that the nowhere-vanishing $2$-form
    $$
    \sigma':=du\wedge dx+dv\wedge dy+dw\wedge dz\in H^2(M',\Omega^2_{M'}).
    $$ 
    is $G$-invariant, hence descends to a symplectic form $\sigma$ on $M$. Then $M$ is a smooth quasi-projective complex symplectic variety of dimension six, $f\colon M \rightarrow X$ is a Lagrangian fibration and $X\cong \C \times (\C^2/G)$ has non-isolated quotient singularities.
\end{eg}

\begin{eg} (Matsushita \cite[Thm.\ 1.9]{Mat15}, see also \cite[Thm.\ 11]{Sch22})

\noindent
    The following example shows that the conclusion in Theorem \ref{thm-main} is wrong without assuming that $M$ is smooth.
    Indeed, let $E_1, E_2$ be two elliptic curves, let $A:=E_1\times E_2$ and let $K_2A:=\mathrm{ker}(\mathrm{Alb}:S^{[3]}A\to A)$ be the generalized Kummer variety, see \cite[Part.\ 7]{Bea83}. Then it is well-known that $K_2A$ is a projective irreducible holomorphic symplectic manifold of dimension $4$, and that it admits a Lagrangian fibrations $f' \colon K_2A\to \mathbb P^2$ induced by the projection $A\to E_1$, see \cite[Sec.\ 2]{Mat15}. Now, Matsushima \cite[Sec.\ 2]{Mat15} describes an $f'$-equivariant action of $G:= \Z/3\Z$ on $K_2A$ by symplectic automorphisms and with isolated fixed points which is induced by the action of $G$ on $A$ determined by
    \[
    [1]:(x,y)\mapsto (x+\tau,y),
    \]
    where $\tau$ is a $3$-torsion point on $E_1$. Then the quotient  is an irreducible holomorphic symplectic variety in the sense of \cite{GKP16, Sch22} and the induced morphism
    \[
    f\colon M \rightarrow X:=\mathbb P^2/G
    \]
    is a Lagrangian fibration. Moreover, $X$ has isolated quotient singularities and is not smooth \cite[Sec.\ 2]{Mat15}.
\end{eg}

\bibliographystyle{alpha}
\bibliography{LF.bib}
\end{document}